\documentclass[platex]{amsart}
\usepackage{amssymb}
\usepackage{braket}
\usepackage{mathrsfs}
\usepackage{ifthen}
\usepackage{graphicx}
\usepackage{tikz}
\usetikzlibrary{patterns}
\usepackage{comment} 
\usepackage{mleftright}
\usepackage{cleveref}
\usepackage[all]{xy}
\usepackage{amscd}
\usepackage{amsrefs}

\theoremstyle{plain}
\newtheorem{theo}{Theorem}[section]
\crefname{theo}{Theorem}{Theorems}
\Crefname{theo}{Theorem}{Theorems}
\newtheorem{prop}[theo]{Proposition}
\crefname{prop}{Proposition}{Propositions}
\Crefname{prop}{Proposition}{Propositions}
\newtheorem{lem}[theo]{Lemma}
\crefname{lem}{Lemma}{Lemmas}
\Crefname{lem}{Lemma}{Lemmas}

\crefname{cor}{Corollary}{Corollaries}
\Crefname{cor}{Corollary}{Corollaries}

\crefname{claim}{Claim}{Claims}
\Crefname{claim}{Claim}{Claims}

\crefname{property}{Property}{Properties}
\Crefname{property}{Property}{Properties}

\crefname{problem}{Problem}{Problems}
\Crefname{problem}{Problem}{Problems}

\theoremstyle{definition}

\crefname{defi}{Definition}{Definitions}
\Crefname{defi}{Definition}{Definitions}

\crefname{notation}{Notation}{Notations}
\Crefname{notation}{Notation}{Notations}

\crefname{convention}{Convention}{Conventions}
\Crefname{convention}{Convention}{Conventions}

\crefname{cond}{Condition}{Conditions}
\Crefname{cond}{Condition}{Conditions}

\crefname{assum}{Assumption}{Assumptions}
\Crefname{assum}{Assumption}{Assumptions}

\theoremstyle{remark}
\newtheorem{rem}[theo]{Remark}
\crefname{rem}{Remark}{Remarks}
\Crefname{rem}{Remark}{Remarks}
\newtheorem{ex}[theo]{Example}
\crefname{ex}{Example}{Examples}
\Crefname{ex}{Example}{Examples}

\crefname{ques}{Question}{Questions}
\Crefname{ques}{Question}{Questions}

\crefname{section}{Section}{Sections}
\Crefname{section}{Section}{Sections}
\crefname{subsection}{Subsection}{Subsections}
\Crefname{subsection}{Subsection}{Subsections}
\crefname{figure}{Figure}{Figures}
\Crefname{figure}{Figure}{Figures}

\newtheorem*{acknowledgement}{Acknowledgement}

\newcommand{\Z}{\mathbb{Z}}

\newcommand{\R}{\mathbb{R}}
\newcommand{\C}{\mathbb{C}}

\newcommand{\quat}{\mathbb{H}}
\newcommand{\CP}{\mathbb{CP}}

\newcommand{\fraks}{\mathfrak{s}}
\newcommand{\frakt}{\mathfrak{t}}

\newcommand{\Diff}{\mathrm{Diff}}
\newcommand{\Homeo}{\mathrm{Homeo}}
\newcommand{\Aut}{\mathrm{Aut}}

\newcommand{\inc}{\hookrightarrow}

\newcommand{\del}{\partial}

\newcommand{\id}{\mathrm{id}}

\newcommand{\Pin}{\operatorname{Pin}}

\newcommand{\tr}{\operatorname{tr}}

\usepackage{todonotes}

\title[Dehn twists on spin 4-manifolds]{Dehn twists and the Nielsen realization problem for spin 4-manifolds}
\author{Hokuto Konno}
\address{Graduate School of Mathematical Sciences, the University of Tokyo, 3-8-1 Komaba, Meguro, Tokyo 153-8914, Japan \\and\\
RIKEN iTHEMS, Wako, Saitama 351-0198, Japan}
\email{konno@ms.u-tokyo.ac.jp}

\begin{document}
\maketitle

\begin{abstract}
We prove that, for a closed oriented smooth spin 4-manifold $X$ with non-zero signature, the Dehn twist about a $(+2)$- or $(-2)$-sphere in $X$ is not homotopic to any finite order diffeomorphism.
In particular, we negatively answer the Nielsen realization problem for each group generated by  the mapping class of a Dehn twist.
We also show that there is a discrepancy between the Nielsen realization problems in the topological category and smooth category for connected sums of copies of $K3$ and $S^{2} \times S^{2}$.
The main ingredients of the proofs are Y. Kato's 10/8-type inequality for involutions and a refinement of it.
\end{abstract}

%\tableofcontents

\section{Main results}
\label{section: main results}

Given a smooth manifold $X$, let $\Diff(X)$ denote the group of diffeomorphisms.
The {\it Nielsen realization problem} asks whether a given finite subgroup $G$ of the mapping class group $\pi_{0}(\Diff(X))$ can be realized as a subgroup of $\Diff(X)$, i.e. whether there exists a (group-theoretic) section $s : G \to \Diff(X)$ of the natural map $\Diff(X) \to \pi_{0}(\Diff(X))$ over $G$.
If there is a section, we say that $G$ is {\it realizable in $\Diff(X)$}.
When $X$ is of $\dim=2$ and oriented closed, which is the classical case of the Nielsen realization problem, Kerckhoff~\cite{Ker83} proved that every $G$ is realizable.

In contrast, Raymond and Scott~\cite{RS77} showed that, in every dimension $\geq 3$, there are nilmanifolds for which the Nielsen realization fails essentially using their non-trivial fundamental groups.
Focusing on dimension 4 and simply-connected manifolds, it was recently proven by Baraglia and the author~\cite{BK21} and Farb and Looijenga~\cite{FL21} that the Nielsen realization fails for $K3$, the underlying smooth 4-manifold of a $K3$ surface.
However, to the best of the author's knowledge, the nilmanifolds in \cite{RS77} and
$K3$ are the only examples of 4-manifolds $X$ that are shown to admit finite subgroups of $\pi_{0}(\Diff(X))$ that are not realizable in $\Diff(X)$.
The purpose of this paper is to expand the list of such 4-manifolds considerably.
In particular, we give infinitely many examples of simply-connected 4-manifolds with distinct intersection forms for which the Nielsen realization fails.

For a general 4-manifold, it is not obvious to find a potential example of non-realizable finite subgroups of the mapping class group.
Following Farb and Looijenga~\cite{FL21}, we consider {\it Dehn twists}, which are sources of interesting examples.
Given a $(+2)$- or $(-2)$-sphere $S$ embedded in a 4-manifold $X$, one has a diffeomorphism $T_{S} : X \to X$ called the Dehn twist, whose mapping class $[T_{S}]$ generates an order 2 subgroup of $\pi_{0}(\Diff(X))$ (see \cref{subsection Dehn twist about pm2spheres}).
Our first main result is:

\begin{theo}
\label{theo: main 1}
Let $X$ be a closed oriented smooth spin 4-manifold with non-zero signature and $S$ be a smoothly embedded 2-sphere in $X$ with $[S]^{2}=2$ or $[S]^{2}=-2$.
Then the Dehn twist $T_{S} : X \to X$ about $S$ is not homotopic to any finite order diffeomorphism of $X$.
In particular, the order 2 subgroup of $\pi_{0}(\Diff(X))$ generated by the mapping class $[T_{S}]$ is not realizable in $\Diff(X)$.
\end{theo}

\cref{theo: main 1} generalizes the case when $X=K3$ due to Farb and Looijenga~\cite[Corollary~1.10]{FL21} (see \cref{rem: comparison with FL} for the comparison), and \cref{theo: main 1} immediately implies that the Nielsen realization fails for quite many 4-manifolds, such as $\#_{m}K3\#_{n}S^{2} \times S^{2}$ with $m>0$ and also some infinitely many examples of irreducible 4-manifolds.
See \cref{ex: intro} for details.

%After Seidel's Ph.D. thesis, Dehn twists have been extensively studied in symplectic topology when $X$ is symplectic and $S$ is Lagrangian (see, e.g., \cite{Sei08}), but in this paper we do not consider symplectic structures on 4-manifolds.

\cref{theo: main 1} makes a striking contrast to a recent result by Lee \cite[Corollary~1.5, Remark~1.7]{Lee22}, which implies that the Dehn twist about every $(\pm2)$-sphere in $\CP^{2} \# n(-\CP^{2})$ with $n \leq 8$ is topologically isotopic (hence homotopic) to a smooth involution.
This means that an analogous statement to \cref{theo: main 1} does not hold for {\it non-spin} 4-manifolds.

Another result of this paper concerns a comparison between the Nielsen realization problems in the topological category and the smooth category.
Let $\Homeo(X)$ denote the group of homeomorphisms of a manifold $X$.
As well as the smooth Nielsen realization, we say that a subgroup $G$ of $\pi_{0}(\Homeo(X))$ is {\it realizable in $\Homeo(X)$} if there is a section $s : G \to \Homeo(X)$ of the natural map $\Homeo(X) \to \pi_{0}(\Homeo(X))$ over $G$.
In \cite[Theorem~1.2]{BK21}, Baraglia and the author showed that some order 2 subgroup of $\pi_{0}(\Diff(K3))$ is not realizable in $\Diff(K3)$, even when the corresponding subgroup in $\pi_{0}(\Homeo(K3))$ is realizable in $\Homeo(K3)$.
We generalize this result to connected sums of copies of $K3$ and $S^{2} \times S^{2}$:

\begin{theo}
\label{theo: top smooth Nielsen}
For $m>0$ and $n \geq 0$, set $X = mK3\#nS^{2} \times S^{2}$.
Then there exists an order 2 subgroup $G$ of $\pi_{0}(\Diff(X))$ with the following properties:
\begin{itemize}
\item The group $G$ is not realizable in $\Diff(X)$.
Moreover, a representative of the generator of $G$ is not homotopic to any finite order diffeomorphism of $X$.
\item  The subgroup $G' \subset \pi_{0}(\Homeo(X))$ defined as the image of $G$ under the natural map $\pi_{0}(\Diff(X)) \to \pi_{0}(\Homeo(X))$ is a non-trivial group, and $G'$ is realizable in $\Homeo(X)$.
\end{itemize}
\end{theo}

In other words, a representative $g \in \Diff(X)$ of the generator of $G$ in \cref{theo: top smooth Nielsen} is not homotopic to any finite order diffeomorphism, although $g^{2}$ is smoothly isotopic to the identity and $g$ is topologically isotopic to some topological involution with non-trivial mapping class.
\cref{theo: top smooth Nielsen} gives also an alternative proof of a result by Baraglia~\cite[Proposition~1.2]{B19} about the realization problem along $\Diff(X) \to \Aut(H_{2}(X;\Z))$  (see \cref{section Additional remarks}).

\cref{theo: main 1,theo: top smooth Nielsen} shall be derived from the following constraint on the induced actions of finite order diffeomorphisms on homology.
Let $\sigma(X)$ denote the signature of an oriented closed 4-manifold $X$ and $b_{+}(X)$ denote the maximal-dimension of positive-definite subspaces of $H_{2}(X;\R)$.
For an involution $\varphi$ on the intersection lattice, we denote by $b_{+}^{\varphi}(X)$ (resp. $b_{-}^{\varphi}(X)$) the maximal-dimension of positive-definite (resp. negative-definite) subspaces of the $\varphi$-invariant part $H_{2}(X;\R)^{\varphi}$, and we set $\sigma^{\varphi}(X) = b_{+}^{\varphi}(X) - b_{-}^{\varphi}(X)$.

\begin{theo}
\label{theo: constraint from Kato}
Let $X$ be a closed oriented smooth 4-manifold with $\sigma(X)<0$, and let $\fraks$ be a spin structure on $X$.
Let $g : X \to X$ be a finite order diffeomorphism that preserves orientation of $X$ and $\fraks$, and
let $\varphi : H_{2}(X;\Z)/\mathrm{Tor} \to H_{2}(X;\Z)/\mathrm{Tor}$ denote the action on homology induced from $g$.
Suppose that $\varphi$ is of order 2 and that $\sigma^{\varphi}(X) \neq \sigma(X)/2$.
Then we have
\begin{align}
\label{eq: necessary condition}
-\frac{\sigma(X)}{16} \leq  b_{+}(X)-b_{+}^{\varphi}(X).
\end{align}
Moreover, if $b_{+}(X)-b_{+}^{\varphi}(X) > 0$, we have
\begin{align*}
-\frac{\sigma(X)}{16} +1 \leq  b_{+}(X)-b_{+}^{\varphi}(X).
\end{align*}
\end{theo}

The main ingredients of the proof of \cref{theo: constraint from Kato} are Y.~Kato's 10/8-type inequality for involutions \cite{Ka17} (\cref{theo: Kato 108}) coming from Seiberg--Witten theory and a refinement of it (\cref{theo: Kato 108 refine}).
This refinement is necessary to show the ``moreover'' part of \cref{theo: constraint from Kato}, which shall be used to obtain the results on Dehn twists (\cref{theo: main 1}) for both 
$(+2)$- and $(-2)$-spheres.

Here is an outline of the contents  of this paper.
In \cref{section Kato's 10/8-type inequality for involutions}, we recall Kato's 10/8-type inequality for a smooth involution on a spin 4-manifold.
In \cref{section A refinement of Kato's inequality}, we give a refinement of Kato's inequality by proving a new Borsuk--Ulam-type theorem using equivariant $K$-theory.
In \cref{section Proof of theo: constraint from Kato}, we prove \cref{theo: constraint from Kato} based on Kato's inequality and the refinement of it in \cref{section A refinement of Kato's inequality}.
\cref{sectionProof of theo: main 1,section Proof of theo: top smooth Nielsen} are devoted to prove \cref{theo: main 1,theo: top smooth Nielsen} respectively.
We finish off this paper by giving remarks on another kind of Dehn twists and other variants of the Nielsen realization problem in \cref{section Additional remarks}.

\section{Kato's 10/8-type inequality for involutions}
\label{section Kato's 10/8-type inequality for involutions}

Henceforth, for an oriented closed 4-manifold $X$, we identify $H_{2}(X)$ with $H^{2}(X)$ via the Poincar\'{e} duality.
For an involution $\iota$ on $X$, we set $b_{+}^{\iota}(X) = b_{+}^{\iota_{\ast}}(X)$, and similarly define $b_{-}^{\iota}(X)$ and $\sigma^{\iota}(X)$.
Note that, if $X$ has non-vanishing signature, all diffeomorphisms of $X$ are orientation-preserving, namely, we have $\Diff(X) = \Diff^{+}(X)$, the group of orientation-preserving diffeomorphisms.

First, we recall the notion of even and odd involutions following \cite{AB68,Bry98}.
Let $X$ be an oriented closed smooth 4-manifold and $\fraks$ be a spin structure on $X$.
Let $\iota : X \to X$ be an orientation-preserving diffeomorphism of order 2, and suppose that $\iota$ preserves (the isomorphism class of) $\fraks$.
Then there are exactly two lifts of $\iota$ to $\fraks$ as automorphisms of the spin structure.
We have either both lifts are of order 2 or both are of order 4.
We say that the involution $\iota$ is {\it of even type} if the lifts are of order 2, and 
say that $\iota$ is {\it of odd type} if the lifts are of order 4.
When the fixed-point set $X^{\iota}$ is non-empty, the codimension of all components of $X^{\iota}$ are the same, which is either 4 or 2, and
the parity of $\iota$ determines which of them arises: $X^{\iota}$ is of codimension-4 if $\iota$ is of even type, and $X^{\iota}$ is of codimension-2 if $\iota$ is of odd type (\cite[Proposition~8.46]{AB68}, see also \cite{Rub95}).

\begin{lem}
\label{lem: even involution Gsignature}
Let $X$ be an oriented closed smooth 4-manifold and $\fraks$ be a spin structure on $X$.
Let $\iota : X \to X$ be an orientation-preserving diffeomorphism of order 2, and suppose that $\iota$ preserves (the isomorphism class of) $\fraks$ and is of even type.
Then we have $\sigma^{\iota}(X) = \sigma(X)/2$.
\end{lem}

\begin{proof}
By Hirzebruch's signature theorem (for example, \cite[Equation~(12), page~177]{HZ74}), 
$\sigma^{\iota}(X)$ can be obtained by adding $\sigma(X)/2$ to contributions from fixed surfaces of $\iota$.
(Note that, for a general involution, the contribution from isolated fixed points is zero.)
However, $X^{\iota}$ does not contain surfaces since $\iota$ is even.
\end{proof}

An important ingredient of this paper is the following 10/8-type constraint on odd smooth involutions, proven by Y.~Kato~\cite{Ka17} using Seiberg--Witten theory and $\Z/4$-equivariant $K$-theory:

\begin{theo}[Kato~{\cite[Theorem~2.3]{Ka17}}]
\label{theo: Kato 108}
Let $(X, \fraks)$ be a smooth closed oriented spin 4-manifold.
Let $\iota : X \to X$ be a smooth orientation-preserving involution, and suppose that $\iota$ preserves $\fraks$ and is of odd type.
Then we have
\begin{align}
\label{eq: Kato's inequality}
-\frac{\sigma(X)}{16} \leq b_{+}(X)-b_{+}^{\iota}(X).
\end{align}
\end{theo}

\begin{rem}
In \cite{Ka17}, the result corresponding to \cref{theo: Kato 108} is stated using a quantity $b_{+}^{I}(X)$, where $I$ acts on $H^{2}(X;\R)$ as $I=-\iota^{*}$.
Noting the Poincar\'{e} duality, it immediately follows that $b_{+}^{I}(X) = b_{+}(X)-b_{+}^{\iota}(X)$.
\end{rem}

\section{A refinement of Kato's inequality}
\label{section A refinement of Kato's inequality}

To deal with Dehn twists about both of $(+2)$- and $(-2)$-spheres in \cref{theo: main 1}, we shall need the following refinement of Kato's inequality (\cref{theo: Kato 108}), which we call the {\it refined Kato's inequality}:

\begin{theo}
\label{theo: Kato 108 refine}
Let $(X, \fraks)$ be a smooth closed oriented spin 4-manifold.
Let $\iota : X \to X$ be a smooth orientation-preserving involution, and suppose that $\iota$ preserves $\fraks$ and is of odd type.
Suppose that $b_{+}(X)-b_{+}^{\iota}(X) > 0$.
Then we have
\begin{align*}
-\frac{\sigma(X)}{16} +1 \leq b_{+}(X)-b_{+}^{\iota}(X).
\end{align*}
\end{theo}

This shall be proven in \cref{subsec: Proof of theo: Kato 108 refine} using a Borsuk--Ulam-type theorem (\cref{theo: Borsuk-Ulam}), which we first give in \cref{subsection Z4-equivariant K-theory}.

\subsection{$\Z/4$-equivariant $K$-theory}
\label{subsection Z4-equivariant K-theory}

To show \cref{theo: Kato 108 refine}, we prove a new Borsuk--Ulam-type theorem using $\Z/4$-equivariant $K$-theory.
As in Kato's argument \cite{Ka17}, the following approach is modeled on Bryan's argument~\cite{Bry98} for $\mathrm{Pin}(2)$-equivariant $K$-theory.
A difference from Kato's argument is that we shall use the structure of the whole representation ring $R(\Z/4)$.

Set $G=\Z/4$ and let $j$ denote a generator: $G=\{1, j, -1, -j\}$.
(The symbol $j$ stands for a unit quaternion $j \in \Pin(2) \subset \quat$, which is a symmetry that the Seiberg--Witten equations admit.)
Let $\C, \C_{+}$ and $\C_{-}$ be complex 1-dimensional representations of $G$ determined by \[
\tr_{j}\C=1,\quad \tr_{j}\C_{+}=i,\quad \tr_{j}\C_{-}=-i,
\]
where $\tr_{j}$ denotes the trace of the action of $j$ and $i=\sqrt{-1}$.
Namely, $\C$ is the trivial 1-dimensional representation, and $\C_{\pm}$ are representations given as $\pm i$-multiplication of the fixed generator of $G$.
Let $\tilde{\R}$ denote a real 1-dimensional representation of $G$ defined through the surjective homomorphism $G \to \Z/2$ and multiplication of $\Z/2=\{\pm1\}$.
Set $\tilde{\C}= \tilde{\R}\otimes_{\R}\C$.
Recall that the complex representation ring $R(G)$ is given by
\begin{align}
\label{eq: R(G)}
R(G) = \Z[t]/(t^{4}-1),
\end{align}
where $t=\C_{+}$.

Here we recall a general fact, which holds for any compact Lie group $G$, called tom  Dieck's formula by Bryan~\cite{Bry98}.
Let $V, W$ be finite-dimensional unitary representations of $G$. 
Let $V^{+}$ denote the one-point compactification of $V$, naturally acted by $G$.
We regard the point at infinity as the base point of $V^{+}$.
Let $f : V^{+} \to W^{+}$ be a pointed $G$-continuous map.
By the equivariant $K$-theoretic Thom isomorphism, we have that $\tilde{K}_{G}(V^{+}), \tilde{K}_{G}(W^{+})$ are free $\tilde{K}_{G}(S^{0})=R(G)$-modules generated by the equivariant $K$-theoretic Thom classes $\tau_{G}^{K}(V), \tau_{G}^{K}(W)$ respectively, and thus one may define the equivariant $K$-theoretic mapping degree $\alpha_{f} \in R(G)$ of $f$ characterized by
\[
f^{\ast}\tau_{G}^{K}(W) = \alpha_{f}\tau_{G}^{K}(V).
\]
For an element $g \in G$, let $V^{g}, W^{g}$ denote the fixed-point set for $g$, and 
let $(V^{g})^{\perp}, (W^{g})^{\perp}$ denote the orthogonal complement of $V^{g}, W^{g}$ in $V, W$ respectively. 
Let $d(f^{g}) \in \Z$ denote the mapping degree, defined using just the ordinary cohomology, of the fixed-point set map $f^{g} : (V^{g})^{+} \to (W^{g})^{+}$.
For $\beta \in R(G)$, define $\lambda_{-1}\beta \in R(G)$ to be $\sum_{i \geq 0}(-1)^{i}\Lambda^{i}\beta$.
tom  Dieck's formula is:

\begin{prop}[{\cite[Proposition~9.7.2]{tD79}}, see also {\cite[Theorem 3.3]{Bry98}}]
\label{tD formula}
In the above setup, we have
\[
\tr_{g}(\alpha_{f}) = d(f^{g})\tr_{g}(\lambda_{-1}((W^{g})^{\perp}-(V^{g})^{\perp})).
\]
\end{prop}

Now we are ready to prove the Borsuk--Ulam-type theorem we need:

\begin{theo}
\label{theo: Borsuk-Ulam}
Let $G=\Z/4$.
For natural numbers $m_{0}, m_{1}, n_{0}, n_{1} \geq 0$ with $m_{0}<m_{1}$,
suppose that there exists a $G$-equivariant pointed continuous map
\begin{align}
\label{eq: f in BU}
f : (\tilde{\C}^{m_{0}} \oplus (\C_{+}\oplus\C_{-})^{n_{0}})^{+}
\to (\tilde{\C}^{m_{1}} \oplus (\C_{+}\oplus\C_{-})^{n_{1}})^{+}
\end{align}
with $f(0)=0$.
Then we have
\begin{align}
\label{eq: BU desired ineq}
n_{0}-n_{1}+1 \leq m_{1}-m_{0}.
\end{align}
\end{theo}

\begin{rem}
This Borsuk--Ulam-type theorem, \cref{theo: Borsuk-Ulam}, may be of independent interest.
Especially, it is worth noting that \cref{theo: Borsuk-Ulam} allows us to give a proof of Furuta's celebrated 10/8-inequality~\cite{Fu01} using only the $\Z/4$-symmetry of the Seiberg--Witten equations, while the original proof used a bigger symmetry, $\Pin(2)$.
See \cref{rem: 10/8 recover} for further comments on this.

Also, \cref{theo: Borsuk-Ulam} generalizes a result by Pfister and Stolz~\cite[THEOREM, page 286]{PS87}, where they proved \cref{theo: Borsuk-Ulam} for the case that $m_{0}=0$ and $n_{1}=0$.
The argument of Pfister--Stolz is also based on $K$-theory, but in a slightly different way from us.
\end{rem}

\begin{proof}[Proof of \cref{theo: Borsuk-Ulam}]
Let $\alpha = \alpha_{f} \in R(G)$ denote the equivariant $K$-theoretic mapping degree of $f$.
We shall obtain constraints on $\alpha$ from the actions of $-1$ and $j$.
First, note that the $(-1)$-fixed point set map for $f$ is given as $f^{-1} : (\tilde{\C}^{m_{0}})^{+} \to (\tilde{\C}^{m_{1}})^{+}$, and thus the assumption $m_{0}<m_{1}$ implies $d(f^{-1})=0$.
Hence it follows from \cref{tD formula} that $\tr_{-1}(\alpha)=0$.
Thanks to the ring structure \eqref{eq: R(G)} of $R(G)$, $\alpha$ can be expressed in the form 
\begin{align}
\label{eq: expression of degree}
\alpha = \sum_{k=0}^{3}a_{k}t^{k},
\end{align}
where $a_{k} \in \Z$.
Since $\tr_{-1}(t)=-1$, it follows that $\tr_{-1}(\alpha) = a_{0}-a_{1}+a_{2}-a_{3}$.
Thus we have
\begin{align}
\label{eq: first eq}
a_{0}-a_{1}+a_{2}-a_{3}=0.
\end{align}

Next, let us consider the $j$-action on $\alpha$.
First, note that $f^{j}$ is just the identity map on $S^{0} = \{0\} \cup \{\infty\}$, and hence $d(f^{j})=1$.
In general, for complex rank 1 (virtual) representations $L_{1}, \ldots, L_{N} \in R(G)$, one has $\lambda_{-1}(\sum_{i=1}^{N}L_{i}) = \prod_{i=1}^{N}\lambda_{-1}L_{i}$.
Thus, again using \cref{tD formula}, we have
\begin{align}
\label{eq: 1st trj}
\begin{split}
\tr_{j}(\alpha)
&= \tr_{j}\left(\lambda_{-1}\left(\tilde{\C}^{m_{1}-m_{0}} 
\oplus (\C_{+}\oplus \C_{-})^{n_{1}-n_{0}}\right)\right)\\
&= \tr_{j}\left(\lambda_{-1}\left((m_{1}-m_{0})t^{2}+(n_{1}-n_{0})t + (n_{1}-n_{0})t^{3}\right)\right)\\
&=\tr_{j}\left( (1-t^{2})^{m_{1}-m_{0}}(1-t)^{n_{1}-n_{0}}(1-t^{3})^{n_{1}-n_{0}} \right)\\
&=(1+1)^{m_{1}-m_{0}}(1-i)^{n_{1}-n_{0}}(1+i)^{n_{1}-n_{0}}
= 2^{m_{1}-m_{0}+n_{1}-n_{0}}.
\end{split}
\end{align}
On the other hand, from the expression \eqref{eq: expression of degree} of $\alpha$, 
we have $\tr_{j}(\alpha) = a_{0}-a_{2} + (a_{1}-a_{3})i$.
Since $\tr_{j}(\alpha) \in \R$ by \eqref{eq: 1st trj}, we have $a_{1}-a_{3}=0$, and this combined with \eqref{eq: first eq} implies that
\begin{align}
\label{eq: 2nd trj}
\tr_{j}(\alpha) = a_{0}-a_{2} = 2(a_{1}-a_{2}).
\end{align}
Since $a_{1}-a_{2} \in \Z$, the desired inequality \eqref{eq: BU desired ineq} follows from \eqref{eq: 1st trj} and \eqref{eq: 2nd trj}.
\end{proof}

Note that the divisibility by 2 over $\Z$ of the right-hand side of \eqref{eq: 2nd trj} contributes to the ``$+1$''-term in the inequality \eqref{eq: BU desired ineq}, which is the source of the refined Kato's inequality.

\subsection{Proof of \cref{theo: Kato 108 refine}}
\label{subsec: Proof of theo: Kato 108 refine}

Now we are ready to prove the refined Kato's inequality:

\begin{proof}[Proof of \cref{theo: Kato 108 refine}]
Set $G=\Z/4$.
Kato  proved in \cite{Ka17} that the odd involution $\iota$ gives rise to an involutive symmetry $I$ on the Seiberg--Witten equations on $(X,\fraks)$, and the complexification of a finite-dimensional approximation of the $I$-invariant part of the Seiberg--Witten equations is a $G$-equivariant pointed continuous map $f$ of the form \eqref{eq: f in BU} with $f(0)=0$, where the natural numbers $m_{0}, m_{1}, n_{0}, n_{1}$ in \eqref{eq: f in BU} satisfy
\begin{align*}
m_{1}-m_{0} = b_{+}(X) - b_{+}^{\iota}(X),\quad
n_{0}-n_{1} = -\frac{\sigma(X)}{16}.
\end{align*}
By the assumption $b_{+}(X) - b_{+}^{\iota}(X)>0$, we may apply \cref{theo: Borsuk-Ulam} to this $f$, and this completes the proof of \cref{theo: Kato 108 refine}.
\end{proof}

\begin{rem}
\label{rem: 10/8 recover}
Furuta's 10/8-inequality \cite{Fu01} was proved using the $\Pin(2)$-symmetry of the Seiberg--Witten equations for a closed spin 4-manifold $X$.
Using our Borsuk--Ulam-type theorem, \cref{theo: Borsuk-Ulam}, we may recover Furuta's 10/8-inequality using only the $\Z/4$-symmetry of the Seiberg--Witten equations as follow.
Note that $G=\Z/4=\left<j\right>$ is a subgroup of $\Pin(2)=S^{1} \cup jS^{1} \subset \quat$.
Restricting the $\Pin(2)$-symmetry to the $\Z/4$-symmetry in Furuta's construction \cite{Fu01}, we have that the complexification of a finite-dimensional approximation of the Seiberg--Witten equations is a $G$-equivariant pointed continuous map $f$ of the form \eqref{eq: f in BU} with $f(0)=0$ for natural numbers $m_{0}, m_{1}, n_{0}, n_{1}$ with
\begin{align*}
m_{1}-m_{0} = b_{+}(X),\quad
n_{0}-n_{1} = -\frac{\sigma(X)}{8}.
\end{align*}
Applying \cref{theo: Borsuk-Ulam} to $f$, we obtain 
\[
-\frac{\sigma(X)}{8} + 1 \leq b_{+}(X),
\]
provided that $b_{+}(X)>0$.
This inequality is equivalent to the 10/8-inequality \cite[Theorem~1]{Fu01}.
\end{rem}

\section{Proof of \cref{theo: constraint from Kato}}
\label{section Proof of theo: constraint from Kato}

Now we are ready to prove \cref{theo: constraint from Kato}:

\begin{proof}[Proof of \cref{theo: constraint from Kato}]
First, we reduce the problem to involutions following \cite[Proof of Corollary~1.10]{FL21}.
Since the subgroup of $\Diff(X)$ generated by $g$ has a surjective homomorphism onto $\left<\varphi\right> \cong \Z/2$, we have that the order of $g$ is even.
Let $2m$ be the order of $g$, then $g^{m}$ is a smooth involution.
Set $\iota = g^{m}$.
Since $g_{\ast} = \varphi$ is of order 2, we have that either $\iota_{\ast} = \varphi$ or  $\iota_{\ast} = \id$.
By the condition that $g^{\ast}\fraks \cong \fraks$, $\iota$ also preserves $\fraks$.

If $\iota_{\ast} = \varphi$, we have $\sigma^{\iota}(X) \neq \sigma(X)/2$ from the assumption that $\sigma^{\varphi}(X) \neq \sigma(X)/2$.
If $\iota_{\ast} = \id$, we have $\sigma^{\iota}(X) \neq \sigma(X)/2$ since we supposed $\sigma(X) \neq 0$.
Thus, in any of these cases, we have $\sigma^{\iota}(X) \neq \sigma(X)/2$ and hence it follows from \cref{lem: even involution Gsignature} that $\iota$ is of odd type.
Hence it follows from Kato's inequality, \cref{theo: Kato 108}, that
\begin{align}
\label{eq: proof of 1.3}
-\frac{\sigma(X)}{16} \leq b_{+}(X)-b_{+}^{\iota}(X) \leq b_{+}(X)-b_{+}^{\varphi}(X).
\end{align}

To see the ``moreover'' part of the \lcnamecref{theo: constraint from Kato}, suppose that $b_{+}(X)-b_{+}^{\iota}(X)>0$, then we can replace the left-hand side of \eqref{eq: proof of 1.3} with $-\sigma(X)/16 +1$ by the refined Kato's inequality, \cref{theo: Kato 108 refine}.
This completes the proof.
\end{proof}

\section{Proof of \cref{theo: main 1}}
\label{sectionProof of theo: main 1}

\subsection{Dehn twists about $(\pm2)$-spheres}
\label{subsection Dehn twist about pm2spheres}

First, we recall 4-dimensional Dehn twists associated with $(\pm2)$-spheres.
We refer readers to a lecture note by Seidel~\cite[Section~2]{Sei08} for details.
While the construction of the Dehn twist in \cite{Sei08} is described for a Lagrangian sphere in a symplectic 4-manifold, which is always a $(-2)$-sphere, the construction works for any $(-2)$-sphere in a general 4-manifold without any change, and it is easy to obtain an analogous diffeomorphism also for a $(+2)$-sphere, described below.

Given a $(-2)$-sphere $S$ in an oriented 4-manifold $X$, namely a smoothly embedded 2-dimensional sphere $S$ with $[S]^{2}=-2$, one may construct a diffeomorphism $T_{S} : X \to X$ called the {\it Dehn twist} about $S$, which is supported in a tubular neighborhood of $S$ in $X$, as follows.
First, note that a tubular neighborhood of $S$ is diffeomorphic to $T^{\ast}S^{2}$ since $S$ is a $(-2)$-sphere, and fix an embedding $T^{\ast}S^{2} \inc X$.
The Dehn twist $T_{S}$ is the extension by the identity of some compactly supported diffeomorphism $\tau$ of $T^{\ast}S^{2}$ called the {\it model Dehn twist}, which is given as the monodromy around the nodal singular fiber of the family $\C^{3} \to \C, (z_{1},z_{2},z_{3}) \mapsto z_{1}^{2}+z_{2}^{2}+z_{3}^{2}$ over the origin of $\C$.
The model Dehn twist $\tau$ acts on the zero-section $S^{2}$ as the antipodal map and $\tau^{2}$ is smoothly isotopic to the identity through compactly supported diffeomorphisms of $T^{\ast}S^{2}$ \cite[Proposition~2.1]{Sei08}.
Hence the induced action of $T_{S}$ on homology is non-trivial,
more precisely, $(T_S)_\ast : H_{2}(X;\Z) \to H_{2}(X;\Z)$ is given as 
\[
(T_S)_\ast(x) = x + (x \cdot [S])[S],
\]
and $T_{S}^{2}$ is smoothly isotopic to the identity.
% See \cite[page 233]{Sei08}
Thus the mapping class $[T_{S}]$ is non-trivial and it generates an order 2 subgroup of $\pi_{0}(\Diff(X))$.

Next, consider the situation that a $(+2)$-sphere $S$ in an oriented 4-manifold $X$ is given.
Then a tubular neighborhood of $S$ is diffeomorphic to $TS^{2}$.
Via an isomorphism between $TS^{2}$ and $T^{\ast}S^{2}$ obtained by fixing a metric on $S^{2}$,
we may implant the model Dehn twist into $X$ as well as the $(-2)$-sphere case above.
We denote by $T_{S} : X \to X$ also this diffeomorphism, and call $T_{S}$ the Dehn twist as well.
This Dehn twist also generates an order 2 subgroup of $\pi_{0}(\Diff(X))$, since the corresponding statement for a $(-2)$-sphere follows just from a property of the model Dehn twist, and the action on $H_{2}(X)$ is given by 
\[
(T_S)_\ast(x) = x - (x \cdot [S])[S].
\]

We note that every Dehn twist preserves every spin structure:

\begin{lem}
\label{lem: Dehn twists preserve spin structures}
Let $X$ be a closed oriented smooth 4-manifold, and suppose that $X$ admits a spin structure $\fraks$.
Let $S$ be a $(+2)$- or $(-2)$-sphere in $X$.
Then the Dehn twist $T_{S}$ preserves $\fraks$.
\end{lem}

\begin{proof}
Recall that $T_{S}$ is just the identity map on the complement of a tubular neighborhood of $S$ in $X$, which is diffeomorphic to the disk cotangent bundle $D(T^{\ast}S^{2})$.
Thus it suffices to show that, given a spin structure $\frakt$ on $\del D(T^{\ast}S^{2}) = S(T^{\ast}S^{2})$, an extension of $\frakt$ to $D(T^{\ast}S^{2})$ is unique.
By the relative obstruction theory for a natural fibration $B(\Z/2) \to B\mathrm{Spin}(4) \to B\mathrm{SO}(4)$,
% For the relative obstruction theory, see e.g. Subsection 3.3 of Hatcher's book ``Vector Bundles & K-Theory'' (in particular around p.102).
% See also the interpretation of a spin structure as a section of $BSpin \to BSO$, see Kirby's ``The topology of 4-manifolds'', section for spin structures
it follows that the extensions of $\frakt$ are classified by $H^{1}(D(T^{\ast}S^{2}), S(T^{\ast}S^{2}); \Z/2)$, which is the trivial group by the mod 2 Thom isomorphism for $T^{\ast}S^{2} \to S^{2}$.
This completes the proof.
\end{proof}

\subsection{Proof of \cref{theo: main 1}}
Now we are ready to prove our main result on Dehn twists:

\begin{proof}[Proof of \cref{theo: main 1}]
By reversing the orientation, we may suppose that $\sigma(X)<0$.
Note that a $(\pm2)$-sphere turns into a $(\mp2)$-sphere if we reverse the orientation of $X$.
First we consider the case that  a $(-2)$-sphere is given in $X$ with $\sigma(X)<0$.
Let $S$ be a $(-2)$-sphere, and let $\varphi$ denote the induced automorphism of $H_{2}(X;\Z)$ from the Dehn twist $T_{S}$.
Let us calculate $b_{+}^{\varphi}$, $b_{-}^{\varphi}$, and $\sigma^{\varphi}$.
As descibed above, $\varphi$ is given by $\varphi(x) = x + (x \cdot [S])[S]$, namely, $\varphi$ acts on $H_{2}(X)$ as the reflection with respect to the orthogonal complement of the subspace generated by $[S]$.
Here the orthogonal complement is with respect to the intersection form, and hence the complement contains a maximal-dimensional positive-definite subspace.
Thus we have 
\[
b_{+}^{\varphi}(X) = b_{+}(X), \quad b_{-}^{\varphi}(X) = b_{-}(X)-1, \quad \sigma^{\varphi}(X) = \sigma(X)+1.
\]
From this we have that $\sigma^{\varphi}(X) \neq \sigma(X)/2$, since we supposed that $\sigma(X) < 0$ and hence $|\sigma(X)| \geq 8$ since $H_{2}(X;\Z)$ is an even lattice.
Moreover, we also have $-\sigma(X)/16 > b_{+}(X)-b_{+}^{\varphi}(X)$ again by $\sigma(X)<0$.
Now the claim of \cref{theo: main 1} for $(-2)$-spheres in $X$ with $\sigma(X)<0$ follows from \cref{theo: constraint from Kato} combined with \cref{lem: Dehn twists preserve spin structures}.

Next, we consider the case that a $(+2)$-sphere $S$ in $X$ with $\sigma(X)<0$ is given.
Note that, as in the $(-2)$-sphere case above, $\varphi=(T_{S})_{\ast}$ is the reflection with respect to the orthogonal complement of the subspace generated by $[S]$, but now $[S]$ has positive self-intersection.
Thus we have 
\[
b_{+}^{\varphi}(X) = b_{+}(X)-1, \quad b_{-}^{\varphi}(X) = b_{-}(X), \quad \sigma^{\varphi}(X) = \sigma(X)-1.
\]
Again because of $|\sigma(X)| \geq 8$, it follows from this that $\sigma^{\varphi}(X) \neq \sigma(X)/2$.
Moreover, we have $b_{+}(X)-b_{+}^{\varphi}(X) = 1 < -\sigma(X)/16+1$.
Now the desired claim follows from the ``moreover'' part of \cref{theo: constraint from Kato} combined with \cref{lem: Dehn twists preserve spin structures}, and this completes the proof.
\end{proof}

Note that, the ``moreover'' part of \cref{theo: constraint from Kato}, which was derived from the refined Kato's inequality (\cref{theo: Kato 108 refine}), was effectively used to deal with $(+2)$-spheres in $X$ with $\sigma(X)<0$ in the above proof of \cref{theo: main 1}.

\begin{rem}
\label{rem: comparison with FL}
For $X=K3$, the above proof of \cref{theo: main 1} gives an alternative proof of \cite[Corollary~1.10]{FL21} by Farb and Looijenga.
They gave two different proofs of \cite[Corollary~1.10]{FL21}, and one of them is based on Seiberg--Witten theory. 
We also used Seiberg--Witten theory, but in a slightly different manner: our proof uses Kato's result \cite{Ka17}, rather than a result due to Bryan \cite{Bry98} used by Farb and Looijenga.

Kato's inequality \eqref{eq: Kato's inequality} is useful to obtain a result for general spin 4-manifolds as in \cref{theo: main 1} not only $K3$. 
This is essentially because $b_{+}$ is replaced with $b_{+}-b_{+}^{\iota}$ in Kato's inequality \eqref{eq: Kato's inequality}.
\end{rem}

\begin{ex}
\label{ex: intro}
\cref{theo: main 1} tells us that quite many spin 4-manifolds $X$ have (many) non-realizable order 2 subgroups of $\pi_{0}(\Diff(X))$.
Indeed, there are many spin 4-manifolds that admit $(\pm2)$-spheres.
For example, $S^{2} \times S^{2}$ admits both of $(+2)$- and $(-2)$-spheres.
A $K3$ surface, more generally, a spin complete intersection surface $M$ admits a $(-2)$-sphere.
Except for $M=S^{2} \times S^{2}$ we have $\sigma(M)<0$ for such $M$, and thus we may apply \cref{theo: main 1} to $M$ and obtain a non-realizable subgroup,  and, of course, we may apply \cref{theo: main 1} also to the connected sum of $M$ with any spin 4-manifold with $\sigma \leq 0$.
(For the fact that $M$ contains a $(-2)$-sphere, see the proof of Theorem~1.5 in \cite[page 255]{Sei08}. In fact, one may find a Lagrangian sphere in $M$, whose self-intersection is always $-2$.
See also \cite[pages 23--24]{GS99} for the topology of $M$, including when a complete intersection is spin.)
\end{ex}

\section{Proof of \cref{theo: top smooth Nielsen}}
\label{section Proof of theo: top smooth Nielsen}

Given an oriented closed simply-connected smooth 4-manifold $X$, let $\Aut(H_{2}(X;\Z))$ denote the automorphism group of $H_{2}(X;\Z)$ equipped with the intersection form.
Since the space of maximal-dimensional positive-definite subspaces of $H^{2}(X;\R)$ is known to be contractible, it makes sense whether a given $\varphi \in \Aut(H_{2}(X;\Z))$ preserves a given orientation of the positive part of $H^{2}(X;\R)$.
Let us recall the following classical fact:

\begin{theo}[\cite{Do90,Matu86,Bor86}]
\label{theo: K3 image}
Let $\Gamma(K3) \subset \Aut(H_{2}(K3;\Z))$ denote the image of the natural map
$\pi_{0}(\Diff(K3)) \to \Aut(H_{2}(K3;\Z))$.
Then $\Gamma(K3)$ is the index 2 subgroup of $\Aut(H_{2}(K3;\Z))$ which consists of automorphisms that preserve a given orientation of $H^{+}(K3)$.
\end{theo}

We shall also use:

\begin{theo}[{\cite[Theorem 1.1]{BK21}}]
\label{theo: K3 section}
There exists a (group-theoretic) section $s : \Gamma(K3) \to \pi_{0}(\Diff(K3))$ of the natural map $\pi_{0}(\Diff(K3)) \to \Aut(H_{2}(K3;\Z))$.
\end{theo}

\begin{proof}[Proof of \cref{theo: top smooth Nielsen}]
First, we recall a construction of a topological involution $f_{K}$ on $K3$ (i.e. order 2 element of $\Homeo(K3)$) in \cite[Section~3]{BK21}.
Let $-E_{8}$ denote the negative-definite $E_{8}$-manifold, namely, simply-connected closed oriented topological 4-manifold whose intersection form is the negative-definite $E_{8}$-lattice.
Let $f_{S} : S^{2} \times S^{2} \to S^{2} \times S^{2}$ be the diffeomorphism defined by $(x,y) \mapsto (y,x)$.
Since $f_{S}$ has non-empty fixed-point set, which is of codimension-2, we can form an equivariant connected sum of three copies of $(S^{2} \times S^{2}, f_{S})$.
Take a point $x_{0}$ of $3S^{2} \times S^{2}$ outside the fixed-pioint set of $\#_{3}f_{S}$, and attach two copies of $-E_{8}$ with $3S^{2} \times S^{2}$ at $x_{0}$ and $(\#_{3}f_{S})(x_{0})$.
Now we have got a topological involution $\tilde{f} : 3S^{2} \times S^{2}\#2(-E_{8}) \to 3S^{2} \times S^{2}\#2(-E_{8})$.
Let $h : K3 \to 3S^{2} \times S^{2}\#2(-E_{8})$ be a homeomorphism obtained from 
Freedman theory \cite{Fre82}, and define $f_{K} : K3 \to K3$ by $f_{K} = h^{-1} \circ \tilde{f} \circ h$, which is a topological involution on $K3$.

Define a topological involution $f : X \to X$ by an equivariant connected sum $f = \#_{m} f_{K} \#_{n}f_{S}$ on $X = mK3 \# nS^{2}\times S^{2}$ along fixed points, which acts on homology as follows.
Recall that $H^{+}(S^{2} \times S^{2})$ is generated by $[S^{2} \times \mathrm{pt}] + [\mathrm{pt} \times S^{2}]$ and $H^{-}(S^{2} \times S^{2})$ is generated by $[S^{2} \times \mathrm{pt}] - [\mathrm{pt} \times S^{2}]$.
Hence $f_{0}$ acts trivially on $H^{+}(S^{2} \times S^{2})$, and acts on $H^{-}(S^{2} \times S^{2})$ by $(-1)$-multiplication, and thus we have
$b_{+}^{f_{S}}(S^{2} \times S^{2}) = 1$, $b_{-}^{f_{S}}(S^{2} \times S^{2}) = 0$,
and hence
\begin{align}
&b_{+}^{f_{K}}(K3) = 3,\quad b_{-}^{f_{K}}(K3) = 8,\quad \sigma^{f_{K}}(K3)=-5, 
\label{eq: invariant part of numerical invariants1}\\
&b_{+}^{f}(X) = 3m+n,\quad b_{-}^{f}(X) = 8m,\quad \sigma^{f}(X)=-5m+n.
\label{eq: invariant part of numerical invariants2}
\end{align}

It follows from \eqref{eq: invariant part of numerical invariants1} that $(f_{K})_{\ast}$ preserves an orientation of $H^{+}(K3)$, and hence $(f_{K})_{\ast}$ lies in $\Gamma(K3)$ by \cref{theo: K3 image}.
Using the section $s : \Gamma(K3) \to \pi_{0}(\Diff(K3))$ given in \cref{theo: K3 section},
set $\Phi = s((f_{K})_{\ast})$.
Then $\Phi$ is a non-trivial element of $\pi_{0}(\Diff(K3))$ of order 2, and hence a representative $g_{K} : K3 \to K3$ of $\Phi$ is a diffeomorphism whose square $g_{K}^{2}$ is smoothly isotopic to the identity.
By smooth isotopy, we may take $g_{K}$ so that $g_{K}$ pointwise fixes a 4-disk in $K3$.
Similarly, we may obtain a diffeomorphism $g_{S}$ of $S^{2} \times S^{2}$ which is smoothly isotopic to $f_{S}$ and which fixes a 4-disk pointwise.
Fixing disjoint disks $D_{1}^{4}, \ldots, D_{m+n}^{4}$ in $S^{4}$,
form a diffeomorphism
\[
g = \#_{m} g_{K} \#_{n}g_{S} : X \to X
\]
by attaching $g_{K}$'s and $g_{S}$'s with $(S^{4}, \id_{S^{4}})$ along the fixed disks of $g_{K}$'s and $g_{S}$'s and $D_{1}^{4}, \ldots, D_{m+n}^{4}$.
It is clear that $g$ is supported outside $S^{4}_{0} := S^{4} \setminus \sqcup_{i=1}^{m+n}D_{i}^{4}$.

We claim that $g^{2}$ is smoothly isotopic to the identity.
First, for a simply-connected closed oriented 4-manifold $M$, let $\Diff(M, D^{4})$ denote the group of diffeomorphisms fixing pointwise an embedded 4-disk $D^{4}$ in $M$.
It follows from \cite[Proposition~3.1]{Gian08} that we have an exact sequence
\[
1 \to \ker{p} \to \pi_{0}(\Diff(M, D^{4})) \xrightarrow{p} \pi_{0}(\Diff(M)) \to 1,
\]
where $p$ is an obvious homomorphism and $\ker{p}$ is isomorphic to either $\Z/2$ or $0$, which is generated by the mapping class of the Dehn twist $\tau_{M}$ along the 3-sphere parallel to the boundary.
Set $\tau_{K} = \tau_{K3}$ and $\tau_{S} = \tau_{S^{2} \times S^{2}}$.
Note that the relative mapping class $[\tau_{K}]_{\del}$ is non-trivial in $\pi_{0}(\Diff(K3, D^{4}))$ by \cite[Proposition~1.2]{KM20}, while $[\tau_{S}]_{\del}$ is trivial since $\tau_{S}$ can be absorbed into the $S^{1}$-action on $S^{2} \times S^{2}$ given by the rotation of one $S^{2}$-component.
Thus we obtain from $[g_{K}]^{2}=1$ and $[g_{S}]^{2}=1$ that $[g_{K}]_{\del}^{2} = [\tau_{K}]_{\del}\neq1$ and $[g_{S}]_{\del}^{2}=1$.
Hence $[g]^{2}$ is the product of the Dehn twists along necks between $m$-copies of $K3$ and $S^{4}_{0}$.

On the other hand, let $\tau_{S^{4}_{0}} : S^{4}_{0} \to S^{4}_{0}$ be the diffeomorphism defined as the simultaneous Dehn twists near all $\del D_{i}^{4}$.
It follows from \cref{lem: S4 rotation} below that $\tau_{S^{4}_{0}}$ is smoothly isotopic to the identity relative to $\del S^{4}_{0}$.
Thus we have $[g]^{2} = [(\tau_{S^{4}_{0}} \# \id_{X \setminus S^{4}_{0}})\circ g^{2}]$.
Note that $\tau_{S^{4}_{0}}$ restricted to the neck between each $K3$ and $S^{4}_{0}$ cancels the Dehn twist $\tau_{K}$, but $\tau_{S^{4}_{0}}$ yields the Dehn twist on each of the necks between $S^{2} \times S^{2}$'s and $S^{4}_{0}$.
As a result, $[g]^{2}$ is the product of the Dehn twists along the necks between all of $S^{2} \times S^{2}$ and $S^{4}_{0}$.
But each of these Dehn twists can be absorbed into the rotation of $S^{2} \times S^{2}$ as above.
Thus we get $[g]^{2}=1$.

Let $G$ be the subgroup of $\pi_{0}(\Diff(X))$ generated by the mapping class $[g]$.
We claim that this group $G$ is the desired one.
First, by construction, we have $g_{\ast} = f_{\ast}$ on $H_{2}(X;\Z)$.
By a theorem of Quinn~\cite{Q86} and Perron~\cite{P86},
this implies that $g$ and $f$ are topologically isotopic to each other.
Thus the image $G'$ of $G$ under the map $\pi_{0}(\Diff(X)) \to \pi_{0}(\Homeo(X))$ lifts to the order 2 subgroup of $\Homeo(X)$ generated by $f$.
Since $G'$ is a non-trivial group as $g$ acts homology non-trivially,
this proves the statement on $G'$ in the \lcnamecref{theo: top smooth Nielsen}.

What remains to prove is that $g$ is not homotopic to any finite order diffeomorphism of $X$.
However, using $g_{\ast}=f_{\ast}$, \eqref{eq: invariant part of numerical invariants2}, and $m>0$, it is straightforward to see that $\varphi=g_{\ast}$ violates the inequality \eqref{eq: necessary condition} and that $\sigma^{\varphi}(X) \neq \sigma(X)/2$.
Thus the desired assertion follows from \cref{theo: constraint from Kato}.
This completes the proof of the \lcnamecref{theo: top smooth Nielsen}.
\end{proof}

The following \lcnamecref{lem: S4 rotation} and how to use it  in the proof of \cref{theo: top smooth Nielsen} were suggested to the author by David Baraglia:

\begin{lem}
\label{lem: S4 rotation}
Let $N >0$ and let $S^{4}_{0}$ be a $N$-punctured 4-sphere, $S^{4}_{0} = S^{4} \setminus \sqcup_{i=1}^{N}D_{i}^{4}$.
Let $\tau_{S^{4}_{0}} : S^{4}_{0} \to S^{4}_{0}$ be the diffeomorphism defined as the simultaneous Dehn twists near all $\del D_{i}^{4}$.
Then $\tau_{S^{4}_{0}}$ is smoothly isotopic to the identity relative to $\del S^{4}_{0}$.
\end{lem}

\begin{proof}
Regard $S^{4}$ as the unit sphere of $\R^{5}=\R^{2} \oplus \R^{3}$, and
let $S^{1}$ act on $S^{4}$ by the standard rotation of $\R^{2}$-component.
The fixed-point set $\Sigma$ of the $S^{1}$-action is given by $S(0 \oplus \R^{3}) \cong S^{2}$.
We may assume that $D_{i}^{4}$ are embedded disks in $S^{4}$ whose centers $p_{i}$ are on $\Sigma$.
Then the normal tangent space $N_{p_{i}}$ of $\Sigma$ at $p_{i}$ in $S^{4}$ is acted by $S^{1}$ as the standard rotation.

Pick a disk $\hat{D}_{i}^{4}$ in $S^{4}$ that contains $D_{i}^{4}$ so that $\hat{D}_{i}^{4} \setminus D_{i}^{4}$ is diffeomorphic to the annulus $S^{3} \times [0,1]$.
Set $\hat{S}_{0}^{4}  = S^{4} \setminus \sqcup_{i=1}^{N}\hat{D}_{i}^{4}$.
The $S^{1}$-action on $S^{4}$ gives rise to an isotopy $\{\varphi_{t}\}_{t \in [0,1]} \subset \Diff(\hat{S}_{0}^{4})$ from $\id_{\hat{S}_{0}^{4}}$ to itself such that $\{\varphi_{t}|_{\del \hat{D}_{i}^{4}}\}_{t}$ gives the homotopically non-trivial loop in $\mathrm{SO}(4) \subset \Diff(S^{3}) \cong \Diff(\del \hat{D}_{i}^{4})$.

On the other hand, recall that the Dehn twist $\tau$ on $S^{3} \times [0,1]$ is defined by $\tau(y,t) = (g(t)\cdot y,t)$, where $g : [0,1] \to \mathrm{SO}(4)$ is the homotopically non-trivial loop in $\mathrm{SO}(4)$.
By definition, $\tau$ is isotopic to $\id_{S^{3} \times [0,1]}$ by an isotopy
\[
\psi_{t} \in \Diff(S^{3} \times [0,1], S^{3} \times \{1\}),
\]
through the diffeomorphism group fixing $S^{3} \times \{1\}$ pointwise,
such that $\{\psi_{t}|_{S^{3} \times \{0\}}\}_{t}$ gives the homotopically non-trivial loop in $\Diff^{+}(S^{3})$.

Let $\psi_{t}^{i}$ be copies of $\psi_{t}$, regarded as isotopies on $\hat{D}_{i}^{4} \setminus D_{i}^{4}$. 
By gluing $\varphi_{t}$ with $\psi_{t}^{i}$ along $\sqcup_{i=1}^{N}\del \hat{D}_{i}^{4}$, we obtain an isotopy from $\tau_{S_{0}^{4}}$ to $\id_{S_{0}^{4}}$ relative to $\del S_{0}^{4}$.
\end{proof}

\begin{rem}
For $X=K3$, the above proof of \cref{theo: top smooth Nielsen} gives a slight alternative proof  of \cite[Theorem~1.2]{BK21}, which used the adjunction inequality rather than Kato's result~\cite{Ka17}.
\end{rem}

\section{Additional remarks}
\label{section Additional remarks}

\subsection{Another kind of Dehn twists}

A kind of Dehn twists different from that in \cref{theo: main 1} is the Dehn twist along an embedded annulus $S^{3} \times [0,1]$ in a 4-manifold, defined using the generator of $\pi_{1}(\mathrm{SO}(4)) \cong \Z/2$, as described in the previous section.
The square of the Dehn twist of this kind is smoothly isotopic to the identity.
Recently, Kronheimer and Mrowka~\cite{KM20} proved that the Dehn twist $\tau$ along the neck of $K3\#K3$ is not smoothly isotopic to the identity, and J.~Lin~\cite{JL20} showed that the extension of $\tau$ to $K3\#K3\#S^{2} \times S^{2}$ by the identity of $S^{2} \times S^{2}$ is also not smoothly isotopic to the identity.
Hence it turns out that these Dehn twists generate order 2 subgroups of the mapping class groups.
We remark that these subgroups also give counterexamples to the Nielsen realization problem:

\begin{prop}
\label{prop: more dehn twists}
We have the following:
\begin{itemize}
\item[(i)] 
Let $\tau$ be the Dehn twist along the neck of $K3\#K3$.
Then the order 2 subgroup of $\pi_{0}(\Diff(K3\#K3))$ generated by the mapping class of $\tau$ is not realized in $\Diff(K3\#K3)$.
\item[(ii)] 
Let $\tau'$ be the extension of $\tau$ by the identity to $K3\#K3\#S^{2} \times S^{2}$.
Then the order 2 subgroup of $\pi_{0}(\Diff(K3\#K3\#S^{2} \times S^{2}))$ generated by the mapping class of $\tau'$ is not realized in $\Diff(K3\#K3\#S^{2} \times S^{2})$.
\end{itemize}
\end{prop}

\begin{proof}
By the result by Matumoto~\cite{Mat92} and Ruberman~\cite{Rub95}, a simply-connected closed spin 4-manifold with non-zero signature does not admit a homologically trivial locally linear involution.
Since the Dehn twist $\tau$ is homologically trivial, the claim of the \lcnamecref{prop: more dehn twists} immediately follows.
\end{proof}

\subsection{Other variants of the realization problem}
Given a manifold $X$ of any dimension, one may also consider the realization problem for {\it infinite} subgroups of $\pi_{0}(\Diff(X))$ along $\Diff(X) \to \pi_{0}(\Diff(X))$ (or along $\Diff^{+}(X) \to \pi_{0}(\Diff^{+}(X))$ when $\Diff(X) \neq \Diff^{+}(X)$).
To answer this problem negatively, several authors developed cohomological obstructions, which can be thought of as descendants  of an argument started by Morita~\cite{Mo87} for surfaces.
In dimension 4, concrete results on the non-realization were obtained in \cite{Tshi15,GKT21} in this direction (see also \cite{Gian09}).
Concretely, Giansiracusa--Kupers--Tshishiku~\cite{GKT21} studied $X=K3$, and Tshishiku~\cite{Tshi15} considered manifolds of any dimension, but especially a result in \cite[Theorem~9.1]{Tshi15} treated 4-manifolds whose fundamental groups are isomorphic to non-trivial lattices, which does not have overlap with 4-manifolds that we considered in this paper.

Another variant of the realization problem is about the realization along the natural map $\Diff^{+}(X) \to \pi_{0}(\Homeo^{+}(X))$ for a subgroup of the image of this map.
If $X$ is a simply-connected 4-manifold, the natural map $\pi_{0}(\Homeo^{+}(X)) \to \Aut(H_{2}(X;\Z))$ is isomorphic \cite{Q86,P86}, and hence this version of realization problem is equivalent to the realization along the map $\Diff^{+}(X) \to \Aut(H_{2}(X;\Z))$, which has been extensively studied by Nakamura~\cite{Naka10}, Baraglia~\cite{B19,Ba21}, and Lee~\cite{Lee21,Lee22}.
As noted in \cref{section: main results}, \cref{theo: top smooth Nielsen} gives an alternative proof of \cite[Proposition~1.2]{B19} about the realization of an involution of $H_{2}(X;\Z)$.
%Also, in contrast to \cref{theo: main 1}, Lee \cite[Corollary~1.5]{Lee22} showed that for a del Pezzo surface, which is diffeomorphic to $\CP^{2} \# n(-\CP^{2})$ for some $n \leq 8$, every order 2 subgroup of $\Aut(H_{2}(X;\Z))$ {\it is} realizable along $\Diff^{+}(X) \to \Aut(H_{2}(X;\Z))$, which implies that the Dehn twist for a $(-2)$-sphere in $X$ is {\it topologically} isotopic to a smooth involution, but it remains unknown whether a Dehn twist is {\it smoothly} isotopic to a smooth involution.

\begin{acknowledgement}
The author thanks Jin Miyazawa and Masaki Taniguchi for stimulating discussions about Kato's work \cite{Ka17}, which helped him to get a feeling about \cite{Ka17}.
The author wishes to thank David Baraglia for pointing out a mistake in the proof of \cref{theo: top smooth Nielsen} in an earlier draft and suggesting a remedy for it based on \cref{lem: S4 rotation}.
The author also wishes to thank Seraphina Eun Bi Lee for explaining her work \cite{Lee21,Lee22}.
The author was partially supported by JSPS KAKENHI Grant Numbers 17H06461, 19K23412, and 21K13785.
\end{acknowledgement}

\bibliographystyle{plain}
\bibliography{mainref}

\end{document}